\documentclass[11pt,twoside,A4paper]{article} 
\usepackage{amssymb,amsthm,amsmath,amscd,amsfonts,txfonts}
\usepackage{xcolor}

\usepackage[tmargin=1.5cm,bmargin=1.5cm,lmargin=1.7cm,rmargin=1.7cm]{geometry}

\theoremstyle{plain}
\newtheorem{lemma}{Lemma}[section]
\newtheorem{proposition}{Proposition}[section]
\newtheorem{theorem}{Theorem}[section]

\theoremstyle{definition}
\newtheorem{definition}{Definition}[section]

\newtheorem{example}{Example}[section]

\theoremstyle{remark}
\newtheorem{remark}{Remark}[section]

\title{\bfseries\scshape{Crossed modules of ternary Leibniz algebras}}
\author{\bfseries\scshape Kol B\'eatrice GAMOU\thanks{E-mail address: \tt{kolbeatrice18@gmail.com}} \\
D\'epartement de Math\'ematiques,\\
Universit\'e Gamal Abdel Nasser, Conakry, Guinea. \\
\bfseries\scshape  Ibrahima BAKAYOKO \thanks{E-mail address: \tt{ibrahimabakayoko27@gmail.com}}\\
D\'epartement de Math\'ematiques,\\ 
Universit\'e de N'Z\'er\'ekor\'e, Guinea.}

\date{} 

\begin{document} 
\maketitle

\tableofcontents

\begin{abstract} 
The aim of this paper is to construct triassociative algebras (from operators), new actions and crossed modules from a given one,
 and to make the connexion between these notions on Leibniz algebras or triassociative algebras and the corresponding notions 
on ternary Leibniz algebras.

\end{abstract} 


\section{Introduction}
Triassociative algebras or associative trialgebras, as they name indicated, are vector space endowed with three associative products and satisfying eight other identities.
They generalize associative and diassociative algebras and appear in algebraic topology and some other branches of mathematics.
 Contrary to certains algebraic structures such as diassociative algebras \cite{PG,BW,BR,JLL,JR1,DY}, triassociative algebras are
 less investigated. Nevertheless, the classification of complex triassociative algebras up to dimension 2 and
the weighted averaging operator are also introduce in \cite{LJ} and are used to construct triassociative algebras.
The cohomology of triassociative algebras are studied in \cite{DY}, while their BiHom-version are done in \cite{EBA}.

Leibniz algebras, also known as Loday algebras, are introduced by Loday in \cite{JLL} as a non-commutative version of Lie algebras.
Various aspects of these algebras are intensively investigated. Among which one can cite \cite{IA}, \cite{JF}, \cite{BDW}, \cite{JM}.
A ternary algebra consists of a linear vector space $V$ endowed with a trilinear map $\mu : V\times V\times V\rightarrow V$ satisfying
 the Leibniz identity. They originate from the work of Jacobson in 1949 in the study of associative algebra $(A, \cdot)$ that
 are closed relative to the ternary operation $[[a, b], c]$, where $[a, b]=a\cdot b-b\cdot a$. 

Mathematical objects are often understood through studying operators defined on them. For instance, in Gallois theory a field
 is studied by its automorphisms, in analysis functions are studied through their derivations,  and in geometry manifolds are studied through 
their vector fields. Fifty years ago, several operators have been found from studies in analysis, probability and combinatorics. Among these
 operators, one can cite, element of centroid \cite{RM2}, averaging operator, Reynolds operator, Leroux's TD operator, Nijenhuis operator and
 Rota-Baxter operator \cite{LG,MD}.

In \cite{W}, Whitehead introduced crossed modules as a model for connected homotopy 2-types. After that crossed modules of many algebraic
 structures bas been studied, such as associative algebras \cite{JKC}, Lie algebras \cite{JKC}, Leibniz \cite{JKC}, diassociative algebra
 \cite{JKC}, groups \cite{TP}, n-Leibniz algebra \cite{JE}, 3-Lie algebras \cite{BW,AKM}, Lie-Rinehart algebras \cite{I}, 
3-Lie-Rinehart algebras \cite{RW}, etc. They appear in many branches of mathematics such as category theory, cohomology of algebraic structures,
 differential geometry and in physics. 

It is a commun fact to associate to a category of a given algebraic structure to the category of another algebraic structure; for example
 associative to Lie algebra, associative trialgebra to ternary Leibniz algebra, Leibniz algebra to ternary Leibniz, etc. The aim of this paper 
is to construct crossed modules of ternary Leibniz algebras in this way.
In section 2,we give preliminaries on triassociative algebras, Leibniz algebras, ternary Leibniz algebras, and give some
constructions (using operators, quotient and direct sum) of these algebras and recall their crossed modules.
Section 3 contains the constructions concerning crossed modules of ternary Leibniz algebras from crossed modules of triassociative algebras, 
crossed modules of Leibniz algebras and another crossed modules of Leibniz algebras via Rota-Baxter operator. Some propreties of these crossed 
mudules are given.

\section{Preliminaries and some results}
In this section, we recall some basic notions and results regarding Leibniz algebra, ternary Leibniz algebras, triassociative algebras.
\subsection{Triassociative crossed modules}
\begin{definition}\label{tad}
 A triassociative algebra is a vector space $A$ equipped with  three binary maps $\dashv, \perp, \vdash :
A\otimes A\rightarrow A$ (called left, middle and right respectively), satisfying the following axioms :
\begin{eqnarray}
(x\dashv y)\dashv z& \stackrel{(1)}{=}&x\dashv(y\dashv z), \quad (x\vdash y)\vdash z\stackrel{(2)}{=}x\vdash(y\vdash z),
 \quad (x\perp y)\perp z\stackrel{(3)}{=}x\perp(y\perp z),\nonumber\\
 (x\dashv y)\dashv z&\stackrel{(4)}{=}&x\dashv(y\vdash z), \quad (x\dashv y)\dashv z\stackrel{(5)}{=}x\dashv(y\perp z), 
\quad (x\vdash y)\dashv z\stackrel{(6)}{=}x\vdash(y\dashv z),\nonumber\\
(x\dashv y)\vdash z&\stackrel{(7)}{=}&x\vdash(y\vdash z),\quad (x\perp y)\vdash z\stackrel{(8)}{=}x\vdash (y\vdash z), 
\quad(x\perp y)\dashv z\stackrel{(9)}{=}x\perp(y\dashv z),\nonumber\\
(x\dashv y)\perp z&\stackrel{(10)}{=}&x\perp(y\vdash z),\quad (x\vdash y)\perp z\stackrel{(11)}{=}x\vdash(y\perp z).\nonumber
\end{eqnarray}
for  all $x, y, z\in A$.
\end{definition}
\begin{example}
 a) Any associative algebra $(A, \cdot)$ is an associate trialgebra with $\cdot=\dashv=\perp=\vdash$.\\
 b) Any associative dialgebra is an associative trialgebra with trivial middle product.\\
c) If $(A, \dashv, \perp, \vdash)$ is an associative  trialgebra, then so is $(A, \dashv', \perp', \vdash')$, where
$$x\dashv' y:= y\vdash x, \quad x\perp' y:= y\perp x, \quad x\vdash' y:= y\dashv x.$$
\end{example}
\begin{example}
See \cite{E,SLE} for other examples.
\end{example}

\begin{definition}
 Let $({A}, \dashv_A, \perp_A, \vdash_A)$ and  $(B,\dashv_B, \perp_B, \vdash_B)$ be two triassociative algebras. A linear map $f : A\rightarrow B$
is said to be a morphism of triassociative algebras if
\begin{eqnarray}
 f(x\dashv_A y)=f(x)\dashv_Bf(y),\quad 
f(x\perp_A y)=f(x)\perp_Bf(y) \quad \mbox{and}\quad 
f(x\vdash_A y)=f(x)\vdash_Bf(y),\nonumber
\end{eqnarray}
 for all $x, y \in {A}.$
\end{definition}

In order to give other structures of triassociative algebras, we give the following definition.
\begin{definition}\label{bk3}
Let $(A, \dashv, \perp, \dashv)$ be an triassociative algebra. Let $\cdot$ stands for the three operations $\dashv, \perp$ and $\vdash$.
 A linear map $\varphi : A\rightarrow A$ is said to be a
\begin{enumerate}
\item [i)] Rota-Baxter operator (of weight $\lambda\in\mathbb{K}$) if
\begin{eqnarray}
 \varphi(x)\cdot \varphi(y) = \varphi\Big(\varphi(x)\cdot y + x\cdot \varphi(y) +\lambda x\cdot y\Big),\nonumber
\end{eqnarray}
\item [ii)] Nijenhuis operator if 
\begin{eqnarray}
 \varphi(x)\cdot \varphi(y) = \varphi\Big(\varphi(x)\cdot y + x\cdot \varphi(y) -\varphi(x\cdot y)\Big), \nonumber
\end{eqnarray}
\item [iii)] Reynolds operator if 
\begin{eqnarray}
 \varphi(x)\cdot \varphi(y) = \varphi\Big(\varphi(x)\cdot y + x\cdot \varphi(y) -\varphi(x)\cdot\varphi(y)\Big),\nonumber
\end{eqnarray}
\end{enumerate}
for all $x, y\in A$.
\end{definition}

The following theorem follows from direct computation (for example see \cite{BA}).
\begin{theorem}
 Let $(A, \dashv, \perp, \vdash)$ be an triassociative algebra and $R$ (resp. $N$ and $P$) denotes a Rota-Baxter (resp. Nijenhuis and Reynolds)
 operator.
Then, $(A, \triangleleft, \triangle, \triangleright)$ is also an associative trialgebra with respect to each of the set of these operations
$$a)\quad \left\{
\begin{array}{rl}
x\triangleleft y&=R(x)\dashv y+x\dashv R(y)+\lambda x\dashv y\\
x\triangle y&=R(x)\perp y+x\perp R(y)+\lambda x\perp y\\
x\triangleright y&=R(x)\vdash y+x\vdash R(y)+\lambda x\vdash y\\ 
\end{array}
\right.,
$$
$$b)\quad \left\{
\begin{array}{rl}
x\triangleleft y&=N(x)\dashv y+x\dashv N(y)-N(x\dashv y)\\
x\triangle y&=N(x)\perp y+x\perp N(y)-N(x\perp y)\\
x\triangleright y&=N(x)\vdash y+x\vdash N(y)- N(x\vdash y)\\ 
\end{array}
\right.,
$$
$$c)\quad \left\{
\begin{array}{rl}
x\triangleleft y&=P(x)\dashv y+x\dashv P(y)-P(x)\dashv P(y)\\
x\triangle y&=P(x)\perp y+x\perp P(y)-P(x)\perp P(y)\\
x\triangleright y&=P(x)\vdash y+x\vdash P(y)- P(x)\vdash P(y)\\ 
\end{array}
\right.,
$$
for all $x, y\in A$. 
Moreover, $R, N, P : (A, \triangleleft, \triangle, \triangleright)\rightarrow (A, \dashv, \perp, \vdash)$ (resp. $N$ and $P$) are morphism of 
triassociative algebras
  and $R$ (resp. $N$ and $P$) is a Rota-Baxter (resp. Nijenhuis and Reynolds) operator on $(A, \triangleleft, \triangle, \triangleright)$.
\end{theorem}
Now, we introduce element of centroid for triassociative algebras.
\begin{definition}
Let $(A, \dashv, \perp, \vdash)$ be a  triassociative algebra. An element of centroid (resp. averaging operator) on $A$  is a linear map
 $\theta : A\rightarrow A$  such that for any $x, y\in A$,
 $$\theta(x\ast y)=\theta(x)\ast y=x\ast \theta(y),\quad
 (\mbox{resp.}\quad\theta(x)\ast\theta(y)=\theta(\theta(x)\ast y)=\theta(x\ast\theta(y))$$
with $\ast= \dashv, \perp, \vdash.$
 \end{definition}

The proof of the following theorem only uses the definition of the corresponding linear map.
\begin{theorem}
 Let $(A, \dashv, \perp, \vdash)$ be a triassociative algebra and let $\theta : A\rightarrow A$ (resp. $\beta : A\rightarrow A$) be an element of cenroid
 (resp. injective averaging operator) on $A$. Then, $A$ is a triassociative algebra with respect to each of the multiplications
$$\quad \left\{
\begin{array}{rl}
x\triangleleft y&=\theta(x)\dashv y\\
x\triangle y&=\theta(x)\perp y\\
x\triangleright y&=\theta(x)\vdash y\\ 
\end{array}
\right.
\quad \mbox{and}\quad
\left\{
\begin{array}{rl}
x\triangleleft y&=\beta(x)\dashv y\\
x\triangle y&=\beta(x)\perp y\\
x\triangleright y&=\beta(x)\vdash y\\ 
\end{array}
\right.,
$$
for any $x, y\in A$.
\end{theorem}
\begin{example}
 Viewed as a triassociative algebra, the 2-dimensional associative algebra with basis $\{e_1, e_2\}$, multiplication
$$e_1e_1=e_2e_1=0, \quad e_1e_2=e_1, e_2e_2=e_2$$
and averaging operator (\cite{LJ}) $\beta(e_1)=ae_2,\; \beta(e_2)=e_2,\; a\in \bf R$, gives the following triassociative algebra
$$e_1\ast e_1=e_2\ast e_1=0, \quad e_1\ast e_2=ae_1,\quad e_2\ast e_2=e_2,$$
with $\ast=\dashv=\perp=\vdash$.
\end{example}

Now, we have the following definition.
\begin{definition}
Let $(A, \dashv, \perp, \vdash)$ be a triassociative algebra. The subset $A_o$ of $A$ is said to be a subalgebra if
$A_o$ is stable under the three multiplications i.e.  $x\dashv y, x\perp y, x\vdash y\in A_o$, for any $x, y\in A_o$. 
\end{definition}
\begin{example}
  If $\varphi : A\rightarrow B$ is a morphism of triassociative dialgebras, the image $Im\varphi$ is a subalgebra of  $B$.
\end{example}
To introduce the quotient concept for triassociative algebra we give the below definition.
\begin{definition}
Let $(A, \dashv, \perp, \vdash)$ be a triassociative algebra. A subalgebra $I$ of $A$ is said to be a :\\
i) left  ideal of $A$ if  $x\dashv y\in I, x\perp y\in I$ and $x\vdash y\in I$, for any $x\in I, y\in A$.\\
ii) right ideal of $A$ if  $y\dashv x\in I, y\perp x\in I$ and $y\vdash x\in I$, for all $x\in I, y\in A$.\\
iii) two sided ideal of $A$ if $x\ast y\in I$ and $y\ast x\in I$ for all $x\in A, y\in I$ with $\ast=\dashv, \perp, \vdash$ ; that is $I$ is both
left and right ideals.
\end{definition}
\begin{example}
 i) Obviously $I=\{0\}$ and $I=A$ are two-sided ideals.\\
ii) If $\varphi : A\rightarrow B$ is a morphism of triassociative algebras, the kernel $Ker\varphi$ is a two sided ideal in $A$.
iii) $Ann(A)=\{x\in A| x\ast y=y\ast x=0, \ast=\dashv, \perp, \vdash, x, y\in A\}$ is an ideal in $A$.
\end{example}
\begin{proposition}\label{P3}
Let $(A, \dashv, \perp, \vdash)$ be a triassociative algebra and $I$ be a two sided ideal of $A$.
 Then, $A/I$ is a triassociative dialgebra with
$$
\overline{x}\;\overline{\dashv}\;\overline{y}:={x\dashv y}+I,\quad \overline{x}\;\overline{\perp}\;
\overline{y}:={x\perp y}+I\quad \mbox{and}\quad
 \overline{x}\;\overline{\vdash}\;\overline{y}:={x\vdash y}+I,
$$
 for all $\overline{x}, \overline{y}\in A/I.$
\end{proposition}
\begin{proof}
We only prove one axiom, the other being done similarly. For any $x, y, z \in {A}$, we have
\begin{eqnarray}
 (\overline{x}\dashv \overline{y})\perp \overline{z}
&=&((x+I)\dashv (y+I))\perp(z+I)=(x\dashv y+I)\perp(z+I)=(x\dashv y)\perp z+I\nonumber\\
&\stackrel{(10)}{=}&x\perp(y\vdash z)+I=(x+I)\perp(y\vdash z+I)\nonumber\\
&{=}&(x+I)\perp((y+I)\vdash (z+I))\nonumber\\
&=&\overline{x}\perp (\overline{y}\vdash \overline{z}).\nonumber
\end{eqnarray}
This ends the proof.
\end{proof}

\begin{proposition}
Let $({A},  \dashv_{A}, \perp_A, \vdash_{A})$ and $({B}, \dashv_{B}, \perp_B, \vdash_{B})$ be two 
 triassociative algebras. Then, the multiplications
\begin{eqnarray}
 (a_1+b_2)\dashv(a_2+b_2) &:=&a_1\dashv_{A}a_2+b_2\dashv_{B}b_2,\nonumber\\
(a_1+b_1)\perp(a_2+b_2) &:=&a_1\perp_{A}a_2+b_{A}\perp_{B}b_2,\nonumber\\
(a_1+b_1)\vdash(a_2+b_2) &:=& a_1\vdash_{A}a_2+b_{A}\vdash_{B}b_2,\nonumber
\end{eqnarray}
give a triassociative algebra structure to $A\oplus B$. 
Moreover, if $\xi : {A}\rightarrow {B}$  is a linear map, then
$$ \xi : ({A}, \dashv_{A}, \perp_A, \vdash_{A}) 
\rightarrow({B}, \dashv_{B}, \perp_B, \vdash_{B})$$
 is a morphism if and only if its graph  $\Gamma_\xi=\{(x, \xi(x)), x\in A\}$
 is a triassociative subalgebra of $({A}\oplus{B}, \dashv, \perp, \vdash)$.
\end{proposition}

Now, we introduce action for triassociative algebras.
\begin{definition}\label{taa}
Let $({A},  \dashv_{A}, \perp_A, \vdash_{A})$ and $({B}, \dashv_{B}, \perp_B, \vdash_{B})$ be two 
 triassociative algebras. An action of $B$ on $A$ is the given of six bilinear maps
\begin{eqnarray}
 \mu_1^\dashv;\ \mu_1^\perp;\ \mu_1^\vdash : A\times B\rightarrow A;
\quad \mu_2^{\dashv};\ \mu_2^\perp;\ \mu_2^\vdash : B\times A\rightarrow A \nonumber
\end{eqnarray}
 such that 66 equalities hold, which are obtained from the equalities in Definition \ref{tad} by taking one variable in $A$
 and two variables in $B$ (33 equalities), and one variable in $B$ and two variables in $A$ (33 more equalities).
\end{definition}

\begin{example}
1) Any $B$-bimodule (\cite{EBA})  $(\mu_2^\ast\equiv0, \ast=\dashv, \perp, \vdash)$ $A$ is an action of $B$ on $A$.
The action is called trivial if all these six bilinear maps are trivial.\\
2) The three multiplications of any triassociative algebra $A$ may be seen as a structure action of $A$ onto itself.\\
3) If $B$ is a subalgebra of a triassociative algebra $A$ (maybe $B=A$) and $I$ is a two-sided ideal in $A$, then the left, middle
 and right products in $A$ yield an action of $B$ on $I$.\\
4) If $I$ is a two sided ideal of $A$, then the left, middle and the right products yield an action of $A$ on $I$. \\
5)
If $ 0\longrightarrow A\stackrel{\sigma}{\longrightarrow}E\stackrel{\pi}{\longrightarrow} B {\longrightarrow}0$ is a split short exact sequence
(\cite{EBA}) of triassociative algebras,
 that is, there exists a morphism of triassociative algebras $\ \varphi : B\longrightarrow E\ $ such that $\pi \varphi = id_B$, then there is an
 action of the dialgebra $B$ on $A$, defined by taking left, middle and right products in the triassociative algebra $E$ :
\begin{eqnarray}
\mu_1^\ast(x, a) = \sigma(x) \ast_E \varphi(a), \quad \mu_2^\ast(a, x) = \varphi(a)\ast_E \sigma(x), \nonumber
\end{eqnarray}
for any $a\in B, x\in L,$ with $\ast=\dashv, \perp, \vdash$.
\end{example}

\begin{proposition}\label{lbsa}
Let $({A},  \dashv_{A}, \perp_A, \vdash_{A})$, $({B}, \dashv_{B}, \perp_B, \vdash_{B})$ be two 
 triassociative algebras and  together with an action of $B$ on $A$. Then, there is a triassociative algebra
 structure on $A\rtimes B$ which consists with vector space $A\oplus B$ and the multiplications
 \begin{eqnarray}
 (x, a)\triangleleft (y, a)&=&(x\dashv_A y+\mu_1^\dashv(x, b)+ \mu_2^\dashv(a, y), a\dashv_B b)\nonumber,\\
(x, a)\triangle (y, b)&=&(x\perp_A y+\mu_1^\perp(x, b)+ \mu_2^\perp(a, y), a\perp_B b)\nonumber,\nonumber\\
(x, a)\triangleright (y, b)&=&(x\vdash_A y+\mu_1^\vdash(x, b)+ \mu_2^\vdash(a, y), a\vdash_B b)\nonumber\nonumber,
 \end{eqnarray}
for any $x, y\in A, a, b\in B$.
\end{proposition}
\begin{proof}
 It is straightforward by using axioms in Definition \ref{tad} and Definition \ref{taa}. 
\end{proof}
\begin{definition}\label{crm}
A crossed module of triassociative algebras is a morphism $\varphi : A\rightarrow B$, together with an action of $B$ on $A$  
 such that 
\begin{eqnarray}
1)\;\varphi\mu_1^\ast(x, a)&=& \varphi(x) \ast_B a,\quad 2)\;\varphi\mu_2^\ast(a, x)=a\ast_B \varphi(x),\quad 
3)\;\mu_1^\ast(x, \varphi(y))= x\ast_B y= \mu_2^\ast(\varphi(x), y), \nonumber
\end{eqnarray}
for all $x, y\in A, a\in B$, where $\ast=\dashv, \perp, \vdash$. 
   \end{definition}

\begin{example}
 1) An inclusion $I\hookrightarrow A$ of an ideal of a triassociative algebra $A$ is a crossed module, where the action of $A$ on $I$ is
 given by the three products in $A$\\
2) For any bimodule $A$ over a triassociative algebra $B$    the trivial map $\ 0 : A\to B\ $ is a crossed module with the action of $B$ on 
(abelian i.e. $x\ast y=y\ast x, \ast=\dashv, \perp, \vdash, x, y\in A$) triassociative algebra $A$.\\
3) Any morphism of triassociative algebras $\varphi : A\to B$ with $A$ abelian and $Im\varphi\subset Ann(B)$, provides 
a crossed module with the trivial action of $B$ on $A$.
\end{example}

\begin{definition}
 A morphism of crossed modules of triassociative algebras $(A\stackrel{\varphi}{\longrightarrow}B)\longrightarrow
 (A'\stackrel{\varphi'}{\longrightarrow}B')$
 is a pair $(\alpha,\beta)$, where $\alpha: A\longrightarrow A'$ and $\beta:B\longrightarrow B'$ are morphisms of triassociative algebras 
satisfying :
\begin{eqnarray}
1)\;\varphi'\alpha&=& \beta\varphi, \quad
2)\;\alpha\mu_1^\ast(x,  a)=\mu_1'^\ast(\alpha(x), \beta(a)), \quad
3)\; \alpha\mu_2^\ast(a, x)= \mu_2'^\ast(\beta(a), \alpha(x))\nonumber
\end{eqnarray}
for any $x\in A, a\in B$, where $\ast=\dashv, \perp, \vdash$.
  \end{definition}
Here is the triassociative version of (\cite{RFC}, Example 1.2.56).
\begin{example}
 Let $(A, B, \varphi)$ be a crossed module of triassociative algebras and $E$ be a triassociative algebra.
1) Given a morphism of triassociative algebras $\psi : B \rightarrow E$ such that $\psi\varphi=0$, $(0, \psi) : (A, B, \varphi) \rightarrow
({0}, E, 0)$ is a morphism of crossed modules.\\
2) Given a morphism of triassociative algebras $\psi : E \rightarrow B, (0, \psi) : ({0}, E, 0) \rightarrow (A, B, \varphi)$ is a
morphism of crossed modules.\\
3) Given a morphism of triassociative algebras $\psi : D \rightarrow E, (\psi\varphi, \psi) : (A, B, \varphi) \rightarrow (E, E, id_E)$ is
a morphism of crossed modules. In particular, $(\varphi, id_B ) : (A, B, \varphi) \rightarrow (B, B, id_B)$ is a
morphism of crossed modules.\\
4) Given a morphism of triassociative algebras $\xi : E \rightarrow A, (\xi, \xi\varphi) : (E, E, id_E ) \rightarrow (A, B, \varphi)$ is
a morphism of crossed modules. In particular, $(id_A , \varphi) : (A, A, id_A ) \rightarrow (A, B, \varphi)$ is a morphism of crossed modules.
\end{example}

\begin{proposition}
  Let $\varphi : A\to B\ $ be a crossed module of triassociative algebras. Then, the following conditions hold:
\begin{enumerate}
 \item [1)] $\ker\varphi\subset Ann(A)$. 
\item [2)] The image of $\varphi$ is an ideal in $B$.
\item [3)] $Im\varphi$ acts trivially on the annuhilator $Ann(A)$, and so trivially on $\ker\varphi$.
\end{enumerate}
\end{proposition}
\begin{proof}
 It comes from Definition \ref{taa}.
\end{proof}

\begin{proposition}
Let $A$ and $B$ be two triassociative algebras. We have :
\begin{enumerate}
 \item [1)]  A morphism  $\varphi: A\to B$ of triassociative algebras is a crossed module if and only if the maps
$$
(\varphi, id_B) : A\rtimes B\ \longrightarrow B\rtimes B  
\quad\mbox{and}\quad
(id_A, \varphi) : A\rtimes A \longrightarrow A\rtimes B
$$
are morphism of triassociative algebras.
\item [2)] If $\varphi : A\to B$ is a crossed module of triassociative algebras. Then the map
$$ A\rtimes B \longrightarrow\ A\rtimes B,\quad (x,a)\mapsto (-x, \varphi(x) +a)$$ 
is a morphism of triassociative algebras.
\end{enumerate}
\end{proposition}
\begin{proof}
 It follows from Definition \ref{crm}.
\end{proof}

The proof of below theorem follows from direct computation.
\begin{proposition}
Let $\varphi : A\to B\ $ be a crossed module of triassociative algebras with action of $B$ on $A$ defined by $\mu_1^\ast$ and $\mu_2^\ast$. 
Let $\beta$ an injective averaging operator on both $A$ and $B$, and  commutes with $\varphi$. Then,
$\varphi : A_\beta\to B_\beta\ $ is a crossed module of triassociative algebras with action of $B_\beta$ on $A_\beta$ defined by 
\begin{eqnarray}
 \mu_{1, \beta}^\ast(x, a)=\mu_1^\ast(\beta(x), a), \quad \mu_{2, \beta}^\ast(a, x)=\mu_2^\ast(a, \beta(x)), \nonumber
\end{eqnarray}
for any $x\in A, a\in B, \ast=\dashv, \perp, \vdash$.
\end{proposition}

\subsection{(Ternary) Leibniz crossed modules}

\begin{definition}\label{}
Let $\mathcal{L}$ be a vector space over a field $\mathbb{K}$. A (right) Leibniz algebra structure on $\mathcal{L}$ is a  bilinear map 
$[-, -] : \mathcal{L}\otimes \mathcal{L}\rightarrow \mathcal{L}$   satisfying
 \begin{eqnarray}
   [[x, y], z]=[x, [y, z]]+[[x, z], y], \nonumber
 \end{eqnarray}
 for all $x, y, z\in  \mathcal{L}$.
\end{definition}

\begin{example}
1) Consider a three dimensional vector space $L$ with standard basis
$\{e_1 , e_2, e_3\}$. The non-nul brackets given by $[e_1 , e_3]=-2e_1, [e_2, e_2 ]= e_1, [e_3, e_2 ]=-[e_2, e_3]= e_2$
define a Leibniz algebra structure on $L$.\\
2) See also \cite{JKC, JMC, MM} for other examples.
\end{example}

\begin{definition}\label{le}
Let $\mathcal{L}$ and $\mathcal{P}$ be two Leibniz algebras. An action of $\mathcal{P}$ on $\mathcal{L}$  consists of a pair of two bilinear maps, 
 \begin{eqnarray}
\mu_1 : \mathcal{L}\times \mathcal{P}\rightarrow \mathcal{L}, (l, p)\mapsto \mu_1(l, p)\quad\mbox{and}\quad
\mu_2 : \mathcal{P}\times \mathcal{L}\rightarrow \mathcal{L}, (p, l)\mapsto \mu_2(p, l)\nonumber
\end{eqnarray}
 such that
 \begin{eqnarray}
  \mu_1([l_1,l_2],p) & =& [l_1, \mu_1(l_2,p)] + [\mu_1(l_1,p), l_2],\\
  {[\mu_1(l_1, p), l_2]} & =&  [l_1, \mu_2(p,l_2)] + \mu_1([l_1,l_2], p),\\
  \mu_1(\mu_2(p, l_1), l_2) &=& \mu_2(p,[l_1,l_2] ) +\  [\mu_2(p,l_2),l_1],\\
   {[\mu_2(p,l_2),l]} &=& \mu_2(p_1, \mu_2(p_2,l)) + \mu_1(\mu_2(p_1,l), l_2),\\
   \mu_1(\mu_2(p_1, l), l_2) &=& \mu_2(p_1, \mu_1(l,p_2)) +   \mu_2([p_1,p_2],l),\\
  \mu_1(\mu_1(l, p_1), p_2) \ &=& \mu_1(l,[p_1,p_2]) +  \mu_1(\mu_1(l,p_2), p_1),
\end{eqnarray}
for all $l, l_1, l_2\in \mathcal{L}, p, p_1, p_2\in \mathcal{P}$.
\end{definition}

\begin{definition}
A crossed module of Leibniz algebras (or a leibniz crossed module) is a triple $(\mathcal{L},\mathcal{P}, \varphi)$, where 
$\mathcal{L}$ and $\mathcal{P}$ are Leibniz algebras and $\varphi : \mathcal{L}\rightarrow \mathcal{P}$ is a morphism
 of Leibniz algebras  together with an action of $\mathcal{P}$ on $\mathcal{L}$ such that
\begin{eqnarray}
1)\; \varphi\mu_1(l,{p})= [\varphi(l),{p}],\quad 2)\;\varphi\mu_2({p},l)= [{p},\varphi(l)], \quad
 3)\;\mu_2(\varphi(l_{1}),l_{2})= [l_{1},l_{2}]=\mu_1(l_{1},\varphi(l_{2})), \nonumber
\end{eqnarray}
for all $l, l_{1},l_{2}\in \mathcal{L}, {p}\in\mathcal{P}$.
\end{definition}

\begin{definition}\label{cmm}
A morphism of Leibniz crossed module from $(\mathcal{L}, \mathcal{P}, \varphi)$ to $(\mathcal{L}', \mathcal{P}', \varphi')$ is a pair
   $(\alpha, \beta)$ where $\alpha: \mathcal{L}\longrightarrow \mathcal{L}'$ and 
$\beta: \mathcal{P}\longrightarrow \mathcal{P}'$ are morphisms of Leibniz algebras such that
\begin{equation}
 1)\; \beta\varphi = \varphi'\alpha,\quad 2)\;\alpha \mu_1(l,p) =\mu'_1(\alpha(l), \beta(p)),\quad 3)\; \alpha \mu_2(p,l)
=\mu'_2(\beta(p), \alpha(l)), \nonumber
\end{equation}
for all $l\in \mathcal{L},{p} \in \mathcal{P}$.
\end{definition}
\begin{example}
 See \cite{RFC}, Example 1.2.42.
\end{example}


Now we define ternary Leibniz algebras.
\begin{definition}
 A vector space $\mathcal{L}$ endowed with the ternary bracket $[-, -, -]$ is said to be a (right) ternary Leibniz  algebra if 
 the following identity
\begin{eqnarray}
 [[x, y, z], t, u]=[x, y, [z, t, u]]+[x, [y, t, u], z]+[[x, t, u], y, z]\label{lci}\nonumber
\end{eqnarray}
is satisfyied, for any $x, y, z, t, u\in \mathcal{L}$.
\end{definition}

\begin{example}
1) Any two-dimensional vector space with basis $\{e_1, e_2\}$ is a ternary Leibniz for the ternary bracket given by $[e_1, e_2, e_2]=e_1$
 and $0$ otherwise (cf. \cite{LCG}, Theorem 2.14).\\
2) Let $A$ be an associative algebra and $\beta : A\rightarrow A$ an averaging operator on $A$. Then the bracket
$$[a, b, c] = ab \beta(c)-a\beta(c)b-b\beta(c)a + \beta(c)b a$$
 defines a ternary Leibniz algebra structure on $A$.\\
3) Lie triple systems \cite{WL} and Leibniz triple systems \cite{MO}  ternary Leibniz algebras.
\end{example}

\begin{definition}
Let $(\mathcal{L}, [-, -, -])$ be a ternary Leibniz algebra.
1) A subset $L_0$ of is said to be a subalgebra if it is stable under the bracket i.e.
for any $x, y, z\in L_o, [x, y, z]\in L_o$. \\
2) A subalgebra $I$ of $L$ is called the three-sided ideal of $L$ if $[a, x, y], [x, a, y], [x, y, a]\in I$, for any $x, y\in A, a\in I$.\\
3) A morphism of ternary Leibniz algebras is a linear map preserving the ternary bracket i.e., if $(\mathcal{L}', [-, -, -]')$ is another 
ternary Leibniz algebra, the linear map $\varphi : \mathcal{L}\rightarrow\mathcal{L}'$ is said to be a morphism if
$\varphi([x, y, z])=[\varphi(x), \varphi(y), \varphi'(z)]'$, for all $x, y, z\in \mathcal{L}$.
\end{definition}

\begin{definition}
 Let $(\mathcal{L}, [-, -, -])$ be a ternary Leibniz algebra. The linear map $N: \mathcal{L}\rightarrow \mathcal{L}$ is said to a 
Nijenhuis operator on $L$ if
 \begin{eqnarray}
  [N(x), N(y), N(z)]&=&N\Big([N(x), N(y), z]+[N(x), y, N(z)]+[x, N(y), N(z)]\nonumber\\
&&-N\Big([N(x), y, z]+[x, N(y), z]+[x, y, N(z)]\Big)+N^2([x, y, z])\Big)\nonumber,
 \end{eqnarray}
for all $x, y, z\in \mathcal{L}$.
\end{definition}
\begin{theorem}\label{ed}
Let $(\mathcal{L}, [-, -, -])$ be a ternary Leibniz algebra and $N: \mathcal{L}\rightarrow \mathcal{L}$ be a Nijenhuis operator. Then,
\begin{enumerate}
 \item [1)] $\mathcal{L}$ is Leibniz algebra  with respect to the bracket
 \begin{eqnarray}
 [x, y, z]_N&:=&[N(x), N(y), z]+[N(x), y, N(z)]+[x, N(y), N(z)]\nonumber\\
&&-N([N(x), y, z]+[x, y, N(z)]+[x, N(y), z])+N^2([x, y, z])\nonumber,
 \end{eqnarray}
for all $x, y, z\in \mathcal{L}$.
\item [2)] $N$ is a morphism of  $(\mathcal{L}, [-, -, -]_N)$ onto  $(\mathcal{L}, [-, -, -])$.
\item [3)] $N$ is also a Nijenhuis operator on $(\mathcal{L}, [-, -, -]_N)$.
\end{enumerate}
\end{theorem}
\begin{proof}
 The proof is long but straightforward.
\end{proof}

The first part of the below lemma can be found in \cite{JMC, BL}. It asserts that one may associate a ternary Leibniz algebra to a Leibniz
 algebra. 
\begin{lemma}\label{llt}
Let $(\mathcal{L}, [-, -])$ be a Leibniz algebra. Then 
$$ T(\mathcal{L})=(\mathcal{L}, \{x, y, z\}:=[x, [y, z]]),$$
is a ternary Leibniz algebra, for  any $x, y, z\in \mathcal{L}$.\\
Moreover, if $(\mathcal{L}', [-, -]')$ is another Leibniz algebra and $\varphi : \mathcal{L}\rightarrow \mathcal{L}'$ a morphism of Leibniz 
algebras, then $\varphi : T(\mathcal{L})\rightarrow T(\mathcal{L}')$ is a morphism of ternary Leibniz algebras.
\end{lemma}
\begin{proof}
 The second part comes from direct computation.
\end{proof}

As in the case of (binary) Leibniz braket we have :
\begin{proposition}
Let $\mathcal{L}$ be a $\bf K$-vector space, $[-,-,-]$ and $\{-, -, -\}$ two ternary brackets on $\mathcal{L}$.
Then, $(\mathcal{L},[-,-,-])$ is a right Leibniz algebras if, and only if, $(\mathcal{L}, \{-, -, -\})$ is left ternary Leibniz algebra with 
$\left\lbrace x,y,z\right\rbrace =[z,y,x]$.
\end{proposition}
\begin{proof}
For any $x, y, z, t, u\in \mathcal{L}$,
\begin{eqnarray}
 \left\lbrace x,y,\left\lbrace z,t,u\right\rbrace \right\rbrace 
&=& [[u,t,z],y,x]\nonumber\\
&=&[u,t,[z,y,x]]+ [u,[t,y,x],z] + [[u,y,x],t,z]\nonumber\\
&=&\left\lbrace \left\lbrace x,y,z\right\rbrace ,t,u \right\rbrace + \left\lbrace z,\left\lbrace x,y,t\right\rbrace ,u\right\rbrace 
+\left\lbrace z,t,\left\lbrace x,y,u\right\rbrace \right\rbrace.\nonumber
\end{eqnarray}
The converse is proved in the same way. This achieves the proof.
\end{proof}

The below definition is drawn from \cite{BL}.
\begin{definition}
 A linear map $R : \mathcal{L}\rightarrow \mathcal{L}$ on a ternary Leibniz algebra $(\mathcal{L}, [-, -, -]) $ is called a Rota-Baxter operator
 of weight $\lambda\in\mathbb{K}$, if
for any $x, y, z\in \mathcal{L}$,
\begin{eqnarray}
 &&[R(x), R(y), R(z)]=R\Big([R(x), R(y), z]+[R(x), y, R(z)]+[x, R(y), R(z)]\nonumber\\
&&\qquad\qquad\qquad\qquad+\lambda[R(x), y, z]+\lambda[x, R(y), z]+\lambda[x, y, R(z)]+\lambda^2
[x, y, z]\Big).\label{rot}\nonumber
\end{eqnarray}
\end{definition}
The following next two lemmas are stated in \cite{BL} in the color case.
\begin{lemma}\label{trb}
 Given a Rota-Baxter operator $R: L\rightarrow L$ on a ternary Leibniz algebra $L$, we can make $L$ into another ternary Leibniz
algebra with the bracket
\begin{eqnarray}
 &&[x, y, z]_R=[R(x), R(y), z]+[R(x), y, R(z)]+[x, R(y), R(z)]\nonumber\\
&&\qquad\qquad+\lambda[R(x), y, z]+\lambda[x, R(y), z]+\lambda[x, y, R(z)]+\lambda^2
[x, y, z],\nonumber
\end{eqnarray}
for any $x, y, z\in L$.
\end{lemma}
\begin{lemma}\label{lrb}
Let $(\mathcal{L}, [-, -])$ be a Leibniz algebra and $R : \mathcal{L}\rightarrow \mathcal{L}$ a Rota-Baxter operator  of weight $\lambda$.
 Then, $\mathcal{L}_R=(\mathcal{L}, [-, -]_R)$ is also a Leibniz  algebra,  with 
$$[x, y]_R=[R(x), y]+[x, R(y)]+\lambda[x, y],$$
for all $x, y\in\mathcal{L}$.
Moreover, $R$ is a morphism of  $(\mathcal{L}, [-, -]_R)$ onto  $(\mathcal{L}, [-, -])$.
\end{lemma}

\begin{proposition}
Let $(\mathcal{L}, [-, -])$ be a Leibniz algebra. Then, the braket $[-, [-, -]_R]_R$ (Theorem \ref{llt}) and the braket
  $[-, -, -]_R$ (Lemma \ref{trb}) define the same ternary Leibniz algebra structure. More precisely, $[-, -, -]_R=[-, [-, -]_R]_R$.
\end{proposition}
\begin{proof}
By Lemma \ref{llt} and Lemma \ref{lrb}, we have for any $x, y, z\in \mathcal{L}$, 
\begin{eqnarray}
[x,y,z]^{R}
&:=& [x,[y,z]_{R} ]_{R}\nonumber\\
&=&[x, [R(y), z]+[y, R(z)]+\lambda[y, z]]_R\nonumber\\
&=&[R(x), [R(y), z]+[y, R(z)]+\lambda[y, z]]+[x, [R(y), R(z)]]+\lambda[x, [R(y), z]+[y, R(z)]+\lambda[y, z]]\nonumber\\
&=&[R(x), [R(y), z]]+[R(x), [y, R(z)]]+\lambda[R(x), [y, z]]+[x, [R(y), R(z)]]\nonumber\\
&&+\lambda[x, [R(y), z]]+\lambda[x, [y, R(z)]]+\lambda^2[x,[y, z]]\nonumber.
\end{eqnarray}
By Lemma \ref{trb},
\begin{eqnarray}
[x,y,z]^{R}&=& [Rx,Ry,z] +[Rx,y,Rz] +[[x,Ry],Rz]\nonumber\\
&& + \lambda ([[Rx,y],z]+ [[x,Ry],z] + [[x,y],Rz]) + \lambda^{2}[x,y,z].\nonumber\\
&=&[x,y,z]_{R}\nonumber.
\end{eqnarray}
Thus $[x,y,z]^{R}=[x,y,z]_{R}$.
\end{proof}
See \cite{JKC} for the original version of the next definition.
\begin{definition}\label{cd}
Let $\mathcal{L}$ and $\mathcal{P}$ be two ternary Leibniz algebras.We say that $\mathcal{P}$ acts on $\mathcal{L}$ if there exist six bilinear
 maps :
\begin{eqnarray}
m_1 : \mathcal{L}\times \mathcal{P}\times\mathcal{P}\longrightarrow \mathcal{L},\quad
m_2 : \mathcal{P}\times \mathcal{L}\times\mathcal{P}\longrightarrow \mathcal{L},\quad
m_3 : \mathcal{P}\times \mathcal{P}\times\mathcal{L}\longrightarrow \mathcal{L},\nonumber\\
m'_1 : \mathcal{P}\times \mathcal{L}\times\mathcal{L}\longrightarrow \mathcal{L},\quad
m'_2 : \mathcal{L}\times \mathcal{P}\times\mathcal{L}\longrightarrow \mathcal{L},\quad
m'_3 : \mathcal{L}\times \mathcal{L}\times\mathcal{P}\longrightarrow \mathcal{L}\nonumber
\end{eqnarray}
such that 
\footnotesize
\begin{eqnarray}
m_1\left(m_1(l,p_1,p_2),\ p_3, p_4\right)&=& m_1\left(l, p_1, [p_1,p_3,p_4]\right)\ 
+\ m_1\left(l,[p_1, p_3,p_4],\ p_2\right)\ +\ m_1\left(m_1(l, p_3, p_4),\ p_1, p_2\right)\\
\ m_1\left(m_2(p_1,l,p_2),\ p_3, p_4\right)&=& m_2\left(p_1,l [p_2,p_3,p_4]\right)\ 
+\ m_2\left(p_1, m_1(l, p_3,p_4),\ p_2\right)\ +\ m_2\left([p_1,p_3, p_4],\ l, p_2\right)\\
\ m_1\left(m_2(p_1,p_2,l),\ p_3, p_4\right)&=& m_3\left( p_1, p_2,m_1(l,p_3,p_4)\right)\ 
+\ m_3\left(p_1,[p_2, p_3,p_4],\ l\right)\ +\ m_3\left([p_1, p_3, p_4),\ p_2, l\right)\\
\ m_2\left([p_1,p_3, p_4],\ l, p_4\right) &=& m_3\left( p_1, p_2,m_2(p_3,l,p_4)\right)\ 
+\ m_2\left(p_1, m_2(p_3,l,p_4)\right)\ +\ m_1\left(m_2(p_1,l, p_4),\ p_2, p_3\right)\\
\ m_3\left([p_1,p_3, p_4],\ p_4,l\right) &=& m_3\left( p_1, p_2,m_3(p_3,p_4,l)\right)\ 
+\ m_2\left(p_1, m_3(p_2,p_4,l)\right)\ +\ m_1\left(m_3(p_1, p_4,l),\ p_2, p_3\right)\\
\ m_1\left(m'_3(l_1,l_2,p_1),\ p_2, p_3\right)&=& m'_3\left(l_1,l_2, [p_1,p_2,p_3]\right)\ 
+\ \left[ l_1, m_1(l_2, p_2,p_3),\ l_2 \right]\ +\ m'_2\left(m_1(l_1,p_2, p_3),\ l_2, p_1\right)\\
\ m_1\left(m'_2(l_1,p_1,l_2),\ p_2, p_3\right)&=& m'_2\left(l_1,p_1,m_1(l_2,p_2,p_3]\right)\ 
+\ m'_2\left[ l_1[p_2, p_3,p_4],\ l_2 \right]\ +\ m'_2\left(m_1(l,p_2, p_3),\  p_1, l_2\right)\\
 m_3\left(m_1(l_1,p_1,p_2),\ l_2, p_3\right)&=& m'_2\left(l_1,p_1,m_2(p_2,l_2,p_3]\right)
\!\!+\!\! m'_3\left[ l_1,m_2(p_1, l_2,p_3),\ p_2 \right]\!\!+\!\!m_1\left(m'_3(l_1,l_2, p_3),\  p_1, p_2\right)\\
 m'_1\left([p_1,p_2, p_3],\ l_1, l_2\right) &=& m_3\left( p_1, p_2,m'_1(p_3,l_1, l_2)\right) 
\!\!+\!\!m_2\left(p_1, m'_1(p_2,l_1, l_2), p_3\right)\!\!+\!\! m_1\left(m'_1(p_3,l_1, l_2),\ p_2, p_3\right)\\
 m'_2\left(m_1(l_1,p_1,p_2),\ p_3, l_2\right)&=& m'_2\left(l_1,p_1,m_3(p_2,p_3,l_2]\right)
\!\!+\!\! m'_3\left[ l_1,m_3(p_1, p_3,l_2),\ p_2 \right]\!\!+\!\! m_1\left(m'_2(l_1, p_3,l_2),\  p_1, p_2\right)\\
 m_1\left(m'_1(p_1,l_1,l_2),\ p_2, p_3\right)&=& m'_1\left(p_1,l_1,m_1(l_2,p_2,p_3]\right) 
\!\!+\!\!m'_1\left(p_1, m_1(l_2,p_2, p_3),\ l_2 \right)\!\!+\!\!m'_1\left( [p_1,p_2, p_3],\ l_1, l_2\right)\\
 m'_3\left(m_2(p_1,l_1,p_2),\ l_1,p_3\right)&=& m'_2\left(p_1,l_1,m_2(p_2,l_2,p_3)\right)
\!\!+\!\! m'_2\left(p_1,m'_3(l_1,l_2, p_3), p_2 \right)\!\!+\!\!m'_3\left(m_2(p_1,l_2,p_3), l_1, p_2\right)\\
 m'_2\left(m_2(p_1,l_1,p_2),\ p_3, l_2\right)&=& m'_1\left(p_1,l_1,m_3(p_2,p_3,l_2)\right)
\!\!+\!\! m_2\left(p_1,m'_2(l_1, p_3,l_2),\ p_2 \right)\!\!+\!\!m'_3\left(m_3(p_1p_3,l_2),  l_1, l_2\right)\\
 m'_3\left(m_3(p_1,p_2,l_1),\ l_2,p_3\right)&=& m_3\left(p_1,p_2,m'_3(,l_1,l_2,p_3)\right)
 \!\!+\!\! m'_1\left(p_1,m_2(p_2,l_2, p_3),\ l_1 \right)\!\!+\!\! m'_2\left(m_2(p_1,l_2,p_3),\ p_2,l_1\right)\\
 m'_2\left(m_3(p_1,p_2,l_1),\ p_3,l_2\right)&=& m_3\left(p_1,p_2,m'_2(,l_1,p_2,l_2)\right)
\!\!+\!\!\ m'_1\left(p_1,m_3(p_2, p_3,l_2),\ l_1 \right) \!\!+\!\! m'_2\left(m_3(p_1,p_2,l_2),\ p_2,l_1\right)
\end{eqnarray}
for all $l_1, l_2, l_3\in\mathcal{L}, p_1, p_2, p_3\in\mathcal{P}$. The $15$ other equalities are obtained by echanging the role
 of $m_i$ and $m_i'$.
\end{definition}

\begin{example}
 If $\mathcal{L}$ is a ternary Leibniz algebra, $\mathcal{P}$ a ternary Leibniz subalgebra of $\mathcal{L}$ and $\mathcal{Q}$ an three-sided ideal
of $\mathcal{P}$, then the Leibniz bracket in $\mathcal{L}$ yields an action of $\mathcal{P}$ on $\mathcal{Q}$.
\end{example}

\begin{definition}\label{cm}
 A crossed module is a morphism of ternary Leibniz algebras $\varphi: \mathcal{L} \longrightarrow \mathcal{P}$ together with an action of
$\mathcal{P}$ on $\mathcal{L}$ satisfying the following conditions:

\begin{enumerate}
\item [1)] $\varphi(m_1(l,p_1,p_2))= [\varphi(l), p_1,p_2], \quad
 \varphi(m_2(p_1,l,p_2)) =  [p_1,\varphi(l), p_2],\\ 
 \varphi(m_3(p_1,p_2,l))  = [ p_1,p_2,\varphi(l)],\quad$
 $\varphi(m'_1(p,l_1,l_2)) = [p,\varphi(l_1), \varphi(l_2)],\\
  \varphi(m'_2(l_1,p,l_2))  = [\varphi(l_1), p, \varphi(l_2)],\quad
   \varphi(m'_3(l_1,l_2,p))  = [\varphi(l_1), \varphi(l_2),p]$.
\item [2)] $\ (l_1,l_2,l_3) =m'_1(\varphi(l_1),l_2,l_3) = m'_2(l_1, \varphi(l_2),l_3)  = m'_3(l_1,l_2,\varphi(l_3)),$\\
  $m'_3(\varphi(l_1),\varphi(l_2), l_3 ) = m_2(\varphi(l_1),l_2\varphi(l_3)) = m_1(l, \varphi(l_2),\varphi(l_3))$.
\item [3)] $m'_3(l_1,l_2, p) = m_2(\varphi(l_1), l_2,p) = m_3(l_1,\varphi(l_2),p)$, \\
  $m'_2(l_1,p,l_2) = m_3(\varphi(l_1),p, l_2) = m_1(l_1,p,\varphi(l_2))$, \\
  $m'_1(p,l_1,l_2) = m_3(p,\varphi(l_1), l_2) = m_2(p,l_1,\varphi(l_2))$,
\end{enumerate}
for any $l, l_1, l_2\in\mathcal{L},  p, p_1, p_2\in\mathcal{P}$.
\end{definition}

\begin{definition}
A morphism of crossed modules $(\mathcal{L}\stackrel{\varphi}{\longrightarrow}\mathcal{P})\longrightarrow 
(\mathcal{L}'\stackrel{\varphi'}{\longrightarrow}\mathcal{P}')$ of ternary Leibniz algebras is a pair $(\alpha, \beta)$,
 where $\alpha :\mathcal{L}\longrightarrow \mathcal{L}'$ 
and $\beta: \mathcal{P}\longrightarrow \mathcal{P}'$ are
 morphisms of ternary Leibniz algebras such that 
\begin{enumerate}
 \item [1)] $\varphi'\alpha = \beta\varphi$,
 \item [2)]
$\alpha m_1(l,p_1,p_2) = n_1(\alpha(l),\beta(p_1),\beta(p_2)), \quad
\alpha m_2(p_1,l,p_2) = n_2(\beta(p_1),\alpha(l),\beta(p_2)),\\
\alpha m_3(l,p_1,p_2) = n_3(\beta(p_1),\beta(p_2),\alpha(l)),\quad
\alpha m'_1(p,l_1,l_2) = n'_1(\beta(p),\alpha(l_1),\alpha(l_2)),\\
\alpha m'_2(l_1,p,l_2) = n'_2(\alpha(l_1),\beta(p),\alpha(l_2)), \quad
\alpha m'_3(l_1,l_2,p) = n'_3(\alpha(l_1),\alpha(l_2),\beta(p))
$
\end{enumerate}
for any $l, l_1, l_2\in\mathcal{L},  p, p_1, p_2\in\mathcal{P}$.
\end{definition}

Here is the  version of (\cite{RFC}, Example 1.2.42).
\begin{example}
 Let $(\mathcal{L}, \mathcal{P}, \varphi)$ be a ternary Leibniz crossed module and $\mathcal{Q}$ a ternary Leibniz algebra.\\
1) Given a ternary Leibniz morphism $\psi : \mathcal{P} \rightarrow \mathcal{Q}$ such that $\psi\varphi=0, (0, \psi) :
 (\mathcal{L}, \mathcal{P}, \varphi) \rightarrow (\{0\}, \mathcal{Q}, 0)$ is a morphism of ternary Leibniz crossed modules.\\
2) Given a ternary Leibniz morphism $\psi : \mathcal{Q} \rightarrow \mathcal{P}, (0, \psi) : (\{0\}, \mathcal{Q}, 0)
\rightarrow (\mathcal{L}, \mathcal{P}, \varphi)$ is a morphism of ternary Leibniz crossed modules.\\
3) Given a ternary Leibniz morphism $\psi : \mathcal{P}\rightarrow \mathcal{Q}, (\psi\varphi, \psi) : (\mathcal{L}, \mathcal{P}, \varphi)
\rightarrow (\mathcal{Q}, \mathcal{Q}, id_\mathcal{Q})$ is a
morphism of ternary Leibniz crossed modules. In particular, $(\varphi, id_\mathcal{P} ) : (\mathcal{L}, \mathcal{P}, \varphi)
\rightarrow (\mathcal{P}, \mathcal{P}, id_\mathcal{P})$ is a morphism of Leibniz crossed modules.\\
4) Given a ternary Leibniz morphism $\xi : \mathcal{Q}\rightarrow \mathcal{L}, (\xi, \varphi\xi) : (\mathcal{Q}, \mathcal{Q}, id_\mathcal{Q})
\rightarrow(\mathcal{L}, \mathcal{P}, \varphi)$ is a morphism of ternary Leibniz crossed modules. 
In particular, $(id_\mathcal{L} , \varphi) : (\mathcal{L}, \mathcal{L}, id_\mathcal{L})\rightarrow
(\mathcal{L}, \mathcal{P}, \varphi)$ is a morphism of ternary Leibniz crossed modules.
\end{example}
\section{From triassociative and Leibniz crossed modules to ternary Leibniz crossed modules}
We give contructions of crossed modules of ternary Leibniz algebras from either crossed modules of Leibniz algebras, crossed module of
 ternary Leibniz algebras or crossed module of triassociative algebras. Some of their propreties are given.

First, let us recall  the below lemma  and remark (see \cite{JMC}) ; It connects triassociative algebras to ternary Leibniz algebras.
\begin{lemma}\label{att}
Given a triassociative algebra $(A, \dashv, \perp, \vdash)$, the bracket  
\begin{eqnarray}
 [x, y, z]:=x\dashv(y\perp z-z\perp y)-(y\perp z-z\perp y)\vdash x,\nonumber
\end{eqnarray}
makes $A$ into a ternary Leibniz algebra $T(A)$, for  all $x, y, z\in A$.
\end{lemma}
\begin{remark}
 Given a triassociative algebra $(A, \dashv, \perp, \vdash)$, then A is also a ternary algebra with respect to each of the following brackets 
\begin{eqnarray}
 [x, y, z]_1&:=&x\dashv(y\perp z)-(y\perp z)\vdash x\nonumber\\
 {[x, y, z]_2}&:=&(z\perp y)\vdash x-x\dashv(z\perp y),\nonumber
\end{eqnarray}
for  all $x, y, z\in A$.
\end{remark}

Now, we have :
\begin{theorem}\label{fat}
 Let $({A}, {B}, \varphi)$ be a crossed module  of triassociative algebras with action of $B$ on $A$ defined by 
$\mu_1^\dashv, \mu_1^\perp, \mu_1^\vdash : A\times B\rightarrow A, \mu_2^{\dashv}, \mu_2^\perp, \mu_2^\vdash : B\times A\rightarrow A$.
 Let us define the six bilinear maps : 
\begin{eqnarray}
m_1(x, a, b) &=& \mu_1^\dashv(x, a\perp b-b\perp a ) - \mu_2^\vdash(a\perp b-b\perp a, x),\\
m_2(a, x, b) &=& \mu_2^\dashv(a, \mu_1^\perp(x,b)- \mu_2^\perp(b, x)) - \mu_1^\vdash (\mu_1^\perp(x, b)- \mu_2^\perp(b, x), a),\\
m_3(a, b, x) &=& \mu_2^\dashv(a, \mu_2^\perp(b,x)- \mu_1^\perp(x, b)) - \mu_1^\vdash (\mu_2^\perp(b, x)- \mu_1^\perp(x, b), a),\\
m'_1(c, y, z) &=& \mu_2^\dashv(c, y\perp z-z\perp y ) - \mu_1^{\vdash}(y\perp z-z\perp y, c ),\\
m'_2(y, c, z) &=& y\perp(\mu_2^\perp(c, z) -\mu_1^\perp(z, c)) - (\mu_2^\perp(c, z)-\mu_1^\perp(z, c)) \vdash y, \\
m'_3(y, z, c) &=& y\perp(\mu_1^\perp(z, c) -\mu_2^\perp(c, z)) - (\mu_1^\perp(z, c)-\mu_2^\perp(c, z)) \vdash y, 
\end{eqnarray}
for any $x, y, z\in A,   a, b, c\in B$.
\begin{enumerate}
 \item [1)] Then,
\begin{enumerate}
 \item  we have an action of $T(B)$ on $T(A)$.
\item $(T({A}),T({B}), \varphi)$ is crossed module of ternary Leibniz algebras.
\end{enumerate}
\item [2)] Moreover,  if $({A}', {B}', \varphi')$ is another crossed module of triassociative algebras and $(\alpha, \beta)$
is a morphism of crossed module from $({A}, {B}, \varphi)$ to $({A}', {B}', \varphi')$, then $(\alpha, \beta)$ is a morphism of 
crossed module from $(T({A}),T({B}), \varphi)$ to $(T({A}'),T({B}'), \varphi')$.
\end{enumerate}
\end{theorem}
\begin{proof}
  The proof is long and tedious, but straightforward.
\end{proof}

\begin{definition}\cite{JE}\label{ctl}
 Let $\mathcal{L}$ and  $\mathcal{P}$ be two ternary Leibniz algebras with the action of $\mathcal{P}$ on $\mathcal{L}$. 
The semi-direct product $\mathcal{L}\rtimes\mathcal{P}$ is the Leibniz algebra on the vector space $\mathcal{L}\oplus\mathcal{P}$ with the 
bracket
\begin{eqnarray}
 [(l_1, p_1), (l_2, p_2), (l_3, p_3)]&:=&(m_1(l_1, p_1, p_2)+m_2(p_1, l_2, p_3)+m_3(p_1, p_2, l_3)\nonumber\\
&&+m'_1(p_1, l_1, l_2)+m'_2(l_1, p_2, l_3)
+m'_3(l_1, l_2, p_3), [p_1, p_2, p_3]),\nonumber
\end{eqnarray}
for all $l_1, l_2, l_3\in\mathcal{L},   p_1, p_2, p_3\in\mathcal{P}$.
\end{definition}

\begin{theorem}\label{}
 Let $A$ and $B$ be two triassociative algebras together with an action of $B$ on $A$.
Then, $T({A}\rtimes {B})=T({A})\rtimes T({B})$.
\end{theorem}
\begin{proof}
 It follows from Proposition \ref{lbsa}, Theorem \ref{fat} and Definition \ref{ctl}.
\end{proof}

\begin{proposition}
Let $\varphi: {A}\longrightarrow {B}$ be a crossed module  of triassociative algebras. Then  the maps :
\begin{enumerate}
\item [i)] $(\mu, id_{T{(B)}}): T{(A)}\rtimes T{(B)}\longrightarrow T{(B)}\rtimes T{(B)}$,
\item [ii)] $(id_{T{(B)}}, \varphi): T{(A)}\rtimes T{(A)}\longrightarrow T{(A)}\rtimes T{(B)}$,
\item [iii)] $ \varphi: T{(A)}\rtimes T{(B)}\longrightarrow T{(A)}\rtimes T{(B)}$ given by 
$\varphi(x, a)= (-x,\varphi(x) +a)$, 
\end{enumerate}
are morphisms of ternary Leibniz algebras.
\end{proposition}
\begin{proof}
It follows from the Theorem \ref{fat} and  Definition \ref{ctl}.
\end{proof}

\begin{theorem}
Let $ \mathcal{L}$ and $\mathcal{P}$ be two Leibniz algebras together with an action of $\mathcal{P}$ on $\mathcal{L}$ 
defined by the maps $\mu_1 :\mathcal{L}\otimes \mathcal{P}\longrightarrow \mathcal{L}$ and
 $\mu_2 : \mathcal{P}\otimes \mathcal{L}\longrightarrow \mathcal{L}$. 
 Let us define the six linear maps : 
\begin{eqnarray}
m_1: L\times P \times P &\longrightarrow& L, \quad (l, p_1,p_2)\longmapsto  \mu_1(l, [p_1,p_2])\nonumber\\
m_2: P\times L \times P &\longrightarrow& L, \quad ( p_1,l,p_2)\longmapsto  \mu_2(p_2, m_2(l,p_2))\nonumber\\
m_3: P\times P \times L &\longrightarrow& L, \quad (p_1,p_2,l)\longmapsto  \mu_2(p_1, m_2(p_2,l))\nonumber\\
m'_1: P\times L \times L &\longrightarrow& L, \quad ( p,l_1,l_2)\longmapsto  \mu_2(p, [l_1,l_2])\nonumber\\
m'_2: L\times P \times L &\longrightarrow& L,\quad (l_1, p,l_2)\longmapsto \mu_1(l_1,m_2(p,l_2))\nonumber\\
m'_3: L\times L \times P &\longrightarrow& L, \quad ( l_1,l_2,p)\longmapsto  \mu_1(l_1,m_1(l_2, p))\nonumber
\end{eqnarray}
for any $l, l_1, l_2\in\mathcal{L},  p, p_1, p_2\in\mathcal{P}$.
\begin{enumerate}
 \item [1)] Then,
 we have an induced action of $T(\mathcal{P})$ on $T(\mathcal{L})$.
\item [2)] If $(\mathcal{L}, \mathcal{P}, \varphi)$ is a crossed module of Leibniz algebras, then $(T(\mathcal{L}), T(\mathcal{P}), \varphi)$
is a crossed module of ternary Leibniz algebras.
\item [3)] Moreover,  if $(\mathcal{L}', \mathcal{P}', \varphi')$ is another crossed module of Leibniz algebras  and $(\alpha, \beta)$
is a morphism of crossed modules from $(\mathcal{L}, \mathcal{P}, \varphi)$ to $(\mathcal{L}', \mathcal{P}', \varphi')$, 
then $(\alpha, \beta)$ is a morphism of crossed modules from $(T(\mathcal{L}), T(\mathcal{P}), \varphi)$ to 
$(T(\mathcal{L}'), T(\mathcal{P}'), \varphi')$.
\end{enumerate}
\end{theorem}
\begin{proof}
  The proof is long but straightforward.
\end{proof}

\begin{theorem}
 Let $(\mathcal{L}, \mathcal{P}, \varphi)$ be a crossed module of ternary Leibniz algebras with the action of $\mathcal{P}$ on $\mathcal{L}$ 
defined by the maps $\mu_1 :\mathcal{L}\otimes \mathcal{P}\longrightarrow \mathcal{L}$, 
 $\mu_2 : \mathcal{P}\otimes \mathcal{L}\longrightarrow \mathcal{L}$ and $R$ be a Rota-Baxter operator on both  $ \mathcal{L}$ and $\mathcal{P}$
such that for all $l\in \mathcal{L}, p\in \mathcal{P}$,
\begin{eqnarray}
 \mu_1(R(l), R(p))&=&R\Big(\mu_1(R(l), p)+\mu_1(l, R(p))+\lambda\mu_1(l, p)\Big),\nonumber\\
{\mu_2(R(p), R(l))}&=&R\Big(\mu_2(R(p), l)+\mu_2(p, R(l))+\lambda\mu_2(p, l)\Big).\nonumber
\end{eqnarray}
Let us define the bilinear maps $\mu_1^R :\mathcal{L}\otimes \mathcal{P}\longrightarrow \mathcal{L},
 \mu_2^R : \mathcal{P}\otimes \mathcal{L}\longrightarrow \mathcal{L}$
\begin{eqnarray}
 \mu_1^R(l, p)&=&\mu_1(R(l), p)+\mu_1(l, R(p))+\lambda\mu_1(l, p),\nonumber\\
\mu_2^R(p, l)&=&\mu_2(R(p), l)+\mu_2(p, R(l))+\lambda\mu_2(p, l).\nonumber
\end{eqnarray}
Then the triple
\begin{enumerate}
 \item [i)] $(\mathcal{L}_R=(L, [-, -]_R), \mathcal{P}_R=(\mathcal{P}, [-, -]_R),  \varphi)$ is a crossed module of Leibniz algebras with
 action of $\mathcal{P}_R$ on $\mathcal{L}_R$ defined by $\mu_1^R, \mu_2^R$.
\item [ii)] $\Big(T(\mathcal{P}_R)=(T(\mathcal{P}), [-, [-, -]_R]_R), T(\mathcal{L}_R)=(L, [-, [-, -]_R]_R), \varphi\Big)$ is a crossed module
 of ternary Leibniz algebras with action of $T(\mathcal{P}_R)$ on $T(\mathcal{L}_R)$ defined by
\begin{eqnarray}
 m_1(l, p_1,p_2)&=&\mu_1^R(l, [p_1,p_2]_R), \quad m_1'( p,l_1,l_2)=\mu_2^R(p, [l_1,l_2]_R),\nonumber\\
m_2( p_1,l,p_2)&=&\mu_2^R(p_2, m_2^R(l,p_2)),\quad m_2'(l_1, p,l_2)=\mu_1^R(l_1,m_2^R(p,l_2)), \nonumber\\
m_3(p_1,p_2,l)&=&\mu_2^R(p_1, m_2^R(p_2,l)), \quad m_3'(l_1,l_2, p)=\mu_1^R(l_1,m_1^R(l_2, p)), \nonumber
\end{eqnarray}
for any $l, l_1, l_2\in\mathcal{L},  p, p_1, p_2\in\mathcal{P}$.
\end{enumerate}
\end{theorem}

Following Rota-Baxter module \cite{BB1} and Nijenhuis Leibniz representation \cite{BMR}, we have :
 \begin{definition}
 Let $(\mathcal{L}, [-, -, -], R)$  and $(\mathcal{P}, [-, -, -], R)$ be a two ternary Leibniz algebra (i.e. ternary Leibniz algebras equipped
with a Rota-Baxter operator),   $m_i, i=1,2,3,4,5,6$ be an action of $\mathcal{P}$ on $\mathcal{L}$. 
 We say that $R$ is a Rota-Baxter action of ternary Leibniz algebras if, for all $l, l_1, l_2\in \mathcal{L}$,
$p, p_1, p_2\in \mathcal{P}$, we have 
\footnotesize
\begin{eqnarray}
 m_1(R(p), R(l_1), R(l_2))&=&R\Big(m_1(R(p), R(l_1), l_2)+m_1(p, R(l_1), R(l_2))+m_1(R(p), l_1, R(l_2))\Big),\\
m_2(R(l_1), R(p), R(l_2))&=&R\Big(m_2(R(l_1),R(p), l_2)+m_2(l_1, R(p), R(l_2))+m_2(R(l_1), p, R(l_2))\Big),\\
m_3(R(l_1), R(l_2), R(p))&=&R\Big(m_3(R(l_1), l_2, R(p))+m_3(l_1, R(l_2), R(p))+m_3(R(l_1), R(l_2), p)\Big),\\
 m_4(R(l), R(p_1), R(p_2))&=&R\Big(m_4(R(l), R(p_1), p_2)+m_4(l, R(p_1), R(p_2))+m_4(R(l), p_1, R(p_2))\Big),\\
m_5(R(p_1), R(l), R(p_2))&=&R\Big(m_5(R(p_1), R(l), p_2)+m_5(p_1, R(l), R(p_2))+m_5(R(p_1), l, R(p_2))\Big),\\
m_6(R(p_1), R(p_2), R(l))&=&R\Big(m_6(R(p_1), R(p_2), l)+m_6(p_1, R(p_2), R(l))+m_6(R(p_1), p_2, R(l))\Big).
\end{eqnarray}
\end{definition}

\begin{theorem}
 Let $(\mathcal{L}, [-, -, -], R)$  and $(\mathcal{P}, [-, -, -], R)$ be a two ternary Leibniz Rota-Baxter algebras,
$m_i, i=1, 2, 3, 4, 5, 6$ be an action of $\mathcal{P}$ on $\mathcal{L}$ and $R$ a Rota-Baxter action of ternary Leibniz algebras.
Let us define the maps :
\begin{eqnarray}
 m_1'(p, l_1, l_2)&=&m_1(R(p), R(l_1), l_2)+m_1(p, R(l_1), R(l_2))+m_1(R(p), l_1, R(l_2)),\\
m_2'(l_1, p, l_2)&=&m_2(R(l_1),R(p), l_2)+m_2(l_1, R(p), R(l_2))+m_2(R(l_1), p, R(l_2)),\\
m_3'(l_1, l_2, p)&=&m_3(R(l_1), l_2, R(p))+m_3(l_1, R(l_2), R(p))+m_3(R(l_1), R(l_2), p),\\
 m_4'(l, p_1, p_2)&=&m_4(R(l), R(p_1), p_2)+m_4(l, R(p_1), R(p_2))+m_4(R(l), p_1, R(p_2)),\\
m_5'(p_1, l, p_2)&=&m_5(R(p_1), R(l), p_2)+m_5(p_1, R(l), R(p_2))+m_5(R(p_1), l, R(p_2)),\\
m_6'(p_1, p_2, l)&=&m_6(R(p_1), R(p_2), l)+m_6(p_1, R(p_2), R(l))+m_6(R(p_1), p_2, R(l)),
\end{eqnarray}
for all $l, l_1, l_2\in \mathcal{L}$, $p, p_1, p_2\in \mathcal{P}$.
Then $(m'_1, m'_2, m_3', m_4', m_5', m_6')$ is an action of $(\mathcal{P}, [-, -, -]_R)$  on $(\mathcal{L}, [-, -, -]_R)$.
\end{theorem}
\begin{proof}
 It comes from direct computation.
\end{proof}


%
%
\end{document}